\documentclass[a4paper,12pt]{article}
\usepackage{amsfonts,amsmath,array}
\usepackage{amssymb}
\usepackage{amsthm}
\usepackage[T2A]{fontenc}
\usepackage[dvips]{color}
\usepackage[koi8-r]{inputenc}
\usepackage[english]{babel}
\usepackage{amscd}
\usepackage{graphicx}
\usepackage{natbib}
\setcitestyle{authoryear,open={[},close={]}}
\newtheorem{theorem}[]{Theorem}

\newtheorem{cor}{Corollary}
\newtheorem{rem}{Remark}
\newtheorem{definition}{Definition}

\begin{document}
\title
{On the metabelian property of quotient groups of solvable groups of orientation-preserving homeomorphisms of the line
\footnotemark[1]\footnotetext[1]{The reported study was partially funded by RFBR according to the research project 19-01-00147}}
\author{Levon Beklaryan}
\date{Central Economics and Mathematics Institute RAS,\\
Moscow Institute of Physics and Technology\\
Moscow, Russia\\
lbeklaryan@outlook.com, beklar@cemi.rssi.ru}
\maketitle

\begin{abstract}
For the class of solvable groups of homeomorphisms of the line preserving orientation and containing a freely acting element, we establish the metabelianity of the quotient group $G/H_G$, where the elements of the normal subgroup $H_G$ are stabilizers of the minimal set. This fact is an important element in the classification theorem, used, in particular, in the study of the Thompson's group $F$.
\end{abstract}

\section{Introduction}
In what follows, by $Homeo_+(\mathbb X),\mathbb X=\mathbb R,\mathbb S^1$, we denote the group of all orientation-preserving homeomorphisms of $\mathbb X$. We define an important subset $G^S$ of the group $G \subseteq  Homeo_+(\mathbb X)$ as the union of stabilizers
\begin{equation*}
G^S=\bigcup_{t\in \mathbb X}St_G(t).
\end{equation*}
The set $G^s$ may not be a group.

For the group $G\subseteq Homeo_+(\mathbb X),\mathbb X=\mathbb R,\mathbb S^1$, the minimal set is an important topological characteristic.
\begin{definition}
The minimal set of the group $G\subseteq Homeo(\mathbb X)$ is a closed $G$-invariant subset of $\mathbb X$ that does not contain proper closed $G$-invariant subsets. If there is no nonempty minimal set, then by definition we assume that it is empty.
$\blacksquare$
\end{definition}

The importance of minimal sets is that it define supports of metric invariants.

Another major topological characteristic is the set:
\begin{equation*}
Fix\,G^S=\{t\in \mathbb X: \quad \forall g\in G^S, \quad g(t)=t\}.
\end{equation*}

\begin{theorem}[\citep{Beklaryan.2004_RMS}]\label{t1}
Let $G\subseteq Homeo_+(\mathbb X)$. Then one of the following mutually exclusive statements is true:
\begin{itemize}
\item[{a)}] any minimal set is discrete and belongs to the set $Fix~G^S$, and the set $Fix~G^S$ consists of the union of the minimal sets $E_\alpha,\alpha\in \mathcal A$, i.e. $Fix~G^S=\bigcup_{\alpha\in \mathcal A} E_\alpha$;
\item[{b)}] the minimal set is unique; it is a perfect nowhere dense subset of $\mathbb R$; it is contained in the closure of the orbit $\overline{G(t)}$ of an arbitrary point $t\in \mathbb X$ and is denoted by $E(G)$;
\item[{c)}] the minimal set coincides with $\mathbb X$ and is also denoted by $E(G)$;
\item[{d)}] the minimal set is empty (in the case $\mathbb X=\mathbb R$).
\end{itemize}
$\blacksquare$
\end{theorem}
In the case of a {\it discrete minimal set}, we set $E(G)=Fix~sG^S$.

We define one important subgroup of the initial group associated with the minimal set.

\begin{definition}
For the group $G\subseteq Homeo_+(\mathbb X)$, the normal subgroup $H_G$ is defined as follows:
\begin{itemize}
\item[{1)}] if the minimal set is not empty and not discrete, then we set
$$
H_G=\left\{h\in G:\quad E(G)\subseteq Fix<h>\right\};
$$
\item[{2)}] if the minimal set is nonempty and discrete, then we set $H_G=G^S$ (discreteness of the minimal set implies the nonemptiness of the set $Fix~G^S$, the nonemptiness of $Fix~G^S$ implies that $G^S$ is normal subgroup);
\item[{3)}] if the minimal set is empty, then we set $H_G=<e>$.
\end{itemize}
$\blacksquare$
\end{definition}

Let a Radon measure $\mu$ on $\mathbb X$ and a group $G\subseteq Homeo_+(\mathbb X)$ are given. For each element $g \in G$ we determine a new measure $g_*\mu$ defined by the rule: for any Borel set $B$
$$
g_*\mu(B)=\mu(g^{-1}(B)).
$$
If for any $g \in G$ the relation $g_*\mu=c_g\mu,c_g>0$, holds, then such a measure is called projectively invariant. If the equality $c_g=1$ holds for any $g \in G$ then such a measure is called invariant.

Metric invariants are of great importance in the classification of groups of homeomorphisms of the line and the circle. Therefore, the formulation of the criteria for the existence of metric invariants in terms of various characteristics of the group plays an important role.

\section{Main results}
For a group with a freely acting element, we formulate the criterion for the existence of a projectively invariant measure in the form of the absence of some special subgroup with two generators.
\begin{theorem}[\citep{Beklaryan.2014_Sb}]\label{t2}
Let $G\subseteq Homeo_+(\mathbb R)$ and there exists a freely acting element $\bar{g}\in G$ ($Fix~\bar{g}=\emptyset$). For the existence of a projectively invariant measure, it is necessary and sufficient that the group $G$ does not contain the subgroup $\Lambda \subseteq G, \Lambda=<p,q>$, in which the element $p$ is freely acting, and the elements $p,q$ with respect to the points $t_0,t_1 \in E(G), t_0<t_1$, satisfy the conditions:
\begin{align*}
&{q}(t_0)=t_0,\quad q(t_1)=t_1,\quad {q}(t)> t,\quad t\in (t_0,t_1);\\
&p(t_0)\in (t_0,t_1),\quad p(t)>t,\quad t\in \mathbb R.
\end{align*}
$\blacksquare$
\end{theorem}
Below are graphs of the homeomorphisms $p$ and $q$ (see fig.~\ref{fig:pq}).

\begin{figure}[!ht]
\begin{center}
\includegraphics[width=0.75\textwidth]{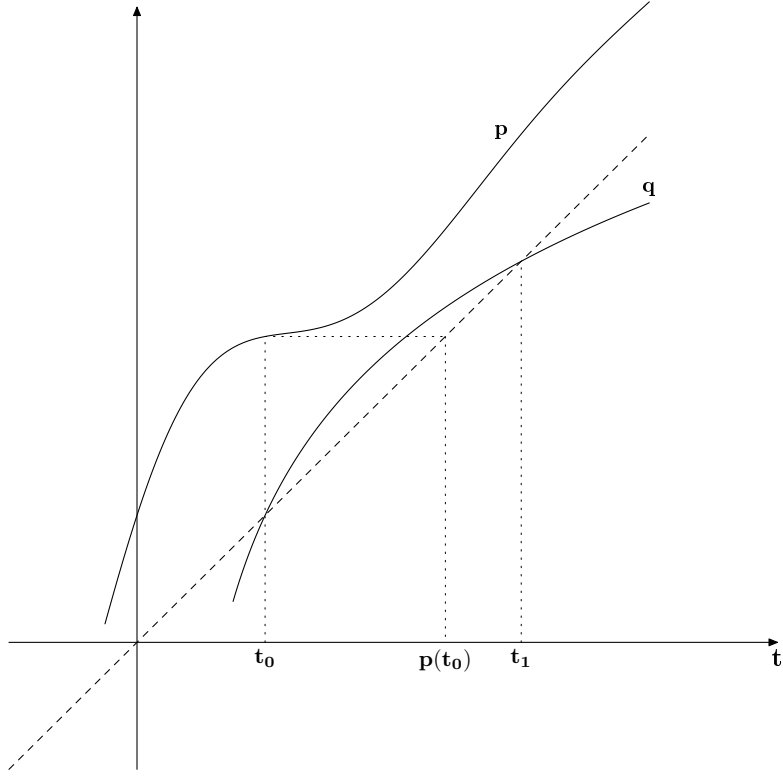}
\caption{Graphs of homeomorphisms $p$ and $q$}
\label{fig:pq}
\end{center}
\end{figure}

In what follows, by $\mathcal L$ we denote the class of groups that contain a freely acting element and do not contain the subgroup $\Lambda=<p,q>$ with two generators from Theorem~\ref{t2}.

\begin{rem}\label{z1}
Every group $G \notin \mathcal L$ contains a free subsemigroup with two generators.
\end{rem}
\begin{proof}
By the condition $G \notin \mathcal L$, for the group $G$ there is a subgroup $ \Lambda=<p,q>$ with two generators described in Theorem~\ref{t2}. Then there is a positive integer $k$ and a point $\bar{t}\in (t_0,t_1)$ such that $p(\bar{t})=q^k(\bar{t})$ and $p(t)>q^k(t), t \in (\bar{t},t_1]$. We form the element $pq^{-k}$. For such an element, the conditions $pq^{-k}(\bar{t})=\bar{t}, pq^{-k}({t})>t, t \in (\bar{t},t_1]$ are fulfilled. Then it follows from \citep{Solodov.1983} that the group $<pq^{-k},q^{-k}>$ contains a free subsemigroup.
\end{proof}

In \citep{Guelman.2016}, a criterion for the existence of a projectively invariant measure is formulated in terms of local subexponentiality with respect to group growth.

\begin{definition}
A group is called locally subexponential if, for any finitely generated subgroup, growth is subexponential.
$\blacksquare$
\end{definition}
By $\mathcal S$ we denote the class of groups for which the following finite normal filtering takes place
$$
<e>\lhd G^{r+1}\lhd G^r \lhd \ldots \lhd G^1\lhd G^0,
$$
where all factors $G_{i-1}/G_i, i=1,\ldots,r+1$, are locally subexponential.

There is a condition of the existence of a projectively invariant measure.
\begin{theorem}[\citep{Guelman.2016}]\label{t3}
Let G be a group in $\mathcal S$ that is acting on the line by order-preserving homeomorphisms. Assume that there is $T \in G$ having no fixed points. Then there is a Radon measure $\mu$ on the line which is quasi-invariant (projectively invariant) by $G$.
$\blacksquare$
\end{theorem}

\begin{cor}\label{c1}
Every solvable group belongs to the class $\mathcal S$. For a solvable group with a freely acting element there exists a projectively invariant measure.
$\blacksquare$
\end{cor}

\begin{rem}\label{z2}
There is the embedding $\mathcal S \subseteq \mathcal L$.
\end{rem}
\begin{proof}
We prove by contradiction. Let there be an element $G \in \mathcal S \backslash \mathcal L$. By definition, for such a group $G$ there exists a subgroup $\Lambda=<p,q>$ with two generators described in Theorem~\ref{t2}, and by Theorem~\ref{t2} for the group $G$ there is no projectively invariant measures. On the other hand, since $G \in \mathcal S$ and there exists a freely acting element for it, by Theorem~\ref{t3} for such a group there exists a projectively invariant measure. This contradiction proves the statement.
\end{proof}

The criteria for the existence of metric invariants can be reformulated in terms of semi-conjugacy, the definition of which is given below.

\begin{definition}
Let groups $G, {}_*G\subseteq Homeo_+(\mathbb R)$, are given. The group $G$ is called {\it semi-conjugate} to the group ${}_*G$ if there exists an orientation-preserving (monotonically increasing) map $\eta:\mathbb R\longrightarrow\mathbb R$ with an image consisting of more than one point, and there exists an epimorphism $\eta^\sharp:G\longrightarrow{}_*G$ such that for any $q\in G$ the diagram
$$
\begin{CD}
\mathbb R @>{}_*g=\eta^\sharp(g)>> \mathbb R\\
@A\eta AA @AA\eta A\\
\mathbb R @>>g> \mathbb R
\end{CD}
$$
is commutative, i.e. $\eta^\sharp(g)\eta=\eta g$.
$\blacksquare$
\end{definition}

In particular, the criterion for the existence of a projectively invariant measure can be reformulated in terms of semi-conjugacy.

\begin{theorem}[\citep{Beklaryan.2004_RMS}]\label{t4}
Let $G\subseteq Homeo_+(\mathbb R)$. Then the following statements are equivalent:
\begin{itemize}
\item[{(1)}] there exists a Borel measure $\mu$ that is finite on compact sets (Radon measure) and projectively invariant with respect to the group $G$;
\item[{(2)}] the group $G$ is semi-conjugate to some group ${}_*G\subseteq Homeo_+(\mathbb R)$ of affine transformations of the line that preserve orientation.
\end{itemize}
$\blacksquare$
\end{theorem}
It was also noted there that, in terms of semi-conjugacy, the canonical subgroup $H_G$ coincides with the kernel of the homomorphism $\eta^\sharp$, i.e. $ker~\eta^\sharp=H_G$, and the quotient group $G/H_G$ is isomorphic to the group of affine transformations of the line ${}_*G$, i.e. ${}_*G\simeq G/H_G$. The group of affine transformations is {\it a solvable group of the solvability length of at most two}.

\begin{cor}\label{c2}
For any group $G$ with a projectively invariant measure and, in particular, for any solvable group $G$ with a freely acting element the quotient group $G/H_G$ is a solvable group of the length of at most two (metabelian group).
\end{cor}
\begin{proof}
By Corollary~\ref{c1}, for such a solvable group there exists a projectively invariant measure. Then the statement for the quotient group follows from Theorem~\ref{t4} and from the property that $ker~\eta^\sharp=H_G$.
\end{proof}

\bibliography{Beklaryan}

\end{document}